\documentclass[12pt,a4paper]{article}
\usepackage[T1]{fontenc}
\usepackage[utf8]{inputenc}
\usepackage[british]{babel}
\usepackage{amsmath,amssymb,amsthm}
\usepackage{tensor}
\usepackage{xcolor}
\usepackage[all]{xy}
\usepackage{cite}
\usepackage{hyperref}

\newcommand{\gen}[1]{\left\langle #1 \right\rangle}
\newcommand{\card}[1]{\lvert #1 \rvert}
\newcommand{\ol}[1]{\overline{#1}}

\let\conj=\induced

\def\N{\mathbb{N}}

\DeclareMathOperator{\Aut}{Aut}
\DeclareMathOperator{\Inn}{Inn}
\DeclareMathOperator{\Hol}{Hol}

\def\B{\mathsf{B}}
\def\L{\mathsf{L}} 
\def\S{\mathsf{S}}

\DeclareMathOperator{\Cent}{C}
\DeclareMathOperator{\Norm}{N}

\newcommand{\id}{\operatorname{id}}

\newtheorem{teo}{Theorem}[section]
\newtheorem{prop}[teo]{Proposition}
\newtheorem{lemma}[teo]{Lemma}
\newtheorem{cor}[teo]{Corollary}

\theoremstyle{definition}
\newtheorem{defi}[teo]{Definition}

\theoremstyle{remark}
\newtheorem{remark}[teo]{Remark}

\theoremstyle{definition}
\newtheorem{ex}[teo]{Example}

\newcommand\puteqnum{\refstepcounter{equation}\textup{(\theequation)}}

\title{Categories of skew left braces and trifactorised groups}
\author{A. Ballester-Bolinches%
\thanks{Departament de Matem\`atiques, Universitat de Val\`encia, Dr.\ Moliner, 50, 46100 Burjassot, Val\`encia, Spain; \texttt{Adolfo.Ballester@uv.es}, \texttt{Ramon.Esteban@uv.es}, \texttt{Pedro.A.Perez@uv.es}, \texttt{Vicent.Perez-Calabuig@uv.es}; ORCID 0000-0002-2051-9075, 0000-0002-2321-8139,  0009-0009-7082-9002, 0000-0003-4101-8656}
\and R. Esteban-Romero\addtocounter{footnote}{-1}\footnotemark \and P. P\'erez-Altarriba\addtocounter{footnote}{-1}\footnotemark \and V. P\'erez-Calabuig\addtocounter{footnote}{-1}\footnotemark}
\begin{document}
\date{}
\maketitle
\begin{abstract}
  The main objective of this paper is to deepen the relationship between skew left braces and trifactorised groups that encodes the information about skew left braces, their structure, their quotients, and their homomorphisms.
  

  \emph{Mathematics Subject Classification (2020):
    16T25, % Yang-Baxter equations 
    81R50, % Quantum groups and related algebraic methods applied to problems in quantum theory
    20C35, % Application of group representations to physics and other areas of science
    20C99, % None of the above, but in this section (20Cxx: Representation theory of groups)
    20D40,  % Products of subgroups of abstract finite groups
  }
\end{abstract}
\section{Introduction}
A \emph{skew left brace} $(B, {+}, {\cdot})$ consists of a set with two binary operations $+$ and~$\cdot$ such that $(B, {+})$ and $(B, {\cdot})$ are groups linked by a distributive-like law $a(b+c)=ab-a+ac$ for $a$, $b$, $c\in B$. As it is usual in this context, the multiplication symbol is omitted, $-a$ denotes the symmetric element of $a$ with respect to~$+$, and $a-b$ denotes $a+(-b)$. When the operations on $B$ are clear from the context, we will refer to a brace as $B$ instead of $(B, {+}, {\cdot})$. Skew left braces were introduced in \cite{GuarnieriVendramin17} and generalise the \emph{left braces} introduced in \cite{Rump07}, in which the first operation is also assumed to be commutative. Skew left braces are interesting algebraic structures, because they represent the exact algebraic framework to study bijective non-degenerate solutions to the
Yang–Baxter equation (see \cite{Yang67,Baxter73,Drinfeld92}). For simplicity, we will use the word \emph{brace} with the meaning of \emph{skew left brace}. 

One of the most important tools in the structural study of braces is the analysis of the trifactorised group associated with the action of the multiplicative group on the additive group. Many brace structural properties can be translated into group structural properties of this trifactorised group which is called here \emph{large trifactorised group} (see \cite{BallesterEsteban22}). Another trifactorised group associated with a brace $B$ appears in a natural way when identifying $B$ with a regular subgroup of the holomorph of the additive group. This group, called here \emph{small trifactorised group}, is extremely useful to construct braces with some prescribed properties (see~\cite{BallesterEstebanJimenezPerezC24-ppt-braces-add-group-trivial-centre}). Small trifactorised groups associated with particular braces lead to construction of interesting examples of groups that can be used to answer some important questions about products of groups (see \cite[Chapter~6]{AmbergFranciosiDeGiovanni92}). 

The aim of this paper is twofold. Firstly, we study the complete family of trifactorised groups associated with a given brace and secondly, we study braces and their associated trifactorised groups from a categorical point of view.  Many of the problems that we discuss in this paper appear in a natural way when considering representations of braces. The following two remarkable results are proved.

\begin{enumerate}
\item There is a bijection between the isomorphism classes of the trifactorised groups associated with a brace $B$ and the orbits of an action of the automorphism group of $B$ over some normal subgroups of the multiplicative group. This result will be useful for computing the trifactorised groups associated with $B$  (Theorem~\ref{teo-orb_3factAss}).

\item  All trifactorised groups associated with a brace $B$ are quotients of the large trifactorised group associated with $B$ and have the small trifactorised group associated with $B$ as a quotient (Theorem~\ref{sql}).

\end{enumerate}

We also consider two main categories: the category $\mathbf{SKB}$ whose objects are braces,  and whose morphisms are homomorphisms of braces, that are maps between braces that preserve both operations, and the category  $\mathbf{3factGrp}$ whose objects are trifactorised groups satisfying some restrictions, and the morphisms are group homomorphisms which are compatible with the triple factorisation. We show here that these two categories are closely related and have interesting properties. It is worth mentioning that some properties of  $\mathbf{SKB}$ have been studied in \cite{BournFacchiniPompili23}. 

%We can find the basic information we need about categories in \cite{Hungerford03} or \cite{MacLane98}.

%The trifactorised groups that have been studied in \cite{BallesterEsteban22} and some trifactorised groups introduced in~\cite{BallesterEstebanJimenezPerezC24-ppt-braces-add-group-trivial-centre} that appear in a natural way when identifying braces with regular subgroups of the holomorph of the additive group, that we call here \emph{small trifactorised groups}, motivate the study of the category $\mathbf{3factGrp}$ of trifactorised groups satisfying some restrictions and its relation with the category $\mathbf{SKB}$.  The most natural candidate for a morphism of trifactorised groups associated to braces is a group homomorphism that is compatible with the triple factorisation.

The contents of the paper are organised as follows. We will begin by introducing two categories in Section~\ref{sec-large-small}, corresponding to the large and small trifactorised groups. Then we will introduce $\mathbf{3factGrp}$ in Section~\ref{sec-3fact}, where we will show that for every brace we can define trifactorised groups satisfying certain conditions and that every trifactorised group with these conditions can be understood as a trifactorised group associated with a brace. We also show how to construct brace homomorphisms from trifactorised group morphisms, and give the exact condition for a brace homomorphism to define a trifactorised group morphism. It turns out that the categories of small and large trifactorised groups are subcategories of $\mathbf{3factGrp}$. Our next objective, developed in Section~\ref{sec-substruc}, is to identify the substructures of a brace~$B$ in a trifactorised group associated with~$B$, and study the substructures of trifactorised groups. We will see that there is a bijection between subbraces and these substructures. We also indicate how to use the trifactorised group to compute images of subgroups of the multiplicative group and preimages of subgroups of the additive group under the associated derivation with the help of the trifactorised group. In Section~\ref{sec-quo}, we characterise the normal subgroups of our trifactorised groups giving quotients that can be regarded as trifactorised groups, and show that they are associated with the corresponding brace quotient. 

The efficiency of algorithms for finite groups in computer algebra systems like \textsf{GAP} \cite{GAP4-14-0}, especially in the case of soluble groups and, in, to a lesser extent, in the case of permutation groups, makes us think that encoding braces as trifactorised groups can help us to simplify computations with braces. Trifactorised groups introduced here are, in general, of smaller order than the groups given in~\cite{BallesterEsteban22}. This can be used to save memory. The saving when using the small trifactorised groups is very relevant in the case of trivial braces, where the small trifactorised group coincides with the additive group.

We also would like to mention that trifactorised groups will be in a forthcoming paper to present a definition of representations of braces over a vector space. We will see there that these brace representations are equivalent to some representations of trifactorised groups over a vector space. It turns out that trifactorised groups are a natural tool to reduce some problems on brace representations to problems on group representations. 

\section{Large and small trifactorised groups}\label{sec-large-small}
Given a brace $B$, there is an action of the multiplicative group $C=(B, {\cdot})$ on the additive group $K=(B, {+})$ called the \emph{lambda action} or the \emph{lambda map} which is defined by $\lambda(a)=\lambda_a$ for every $a\in B$, where $\lambda_a(b)=-a+ab$ for all $a$, $b\in B$. Sysak \cite{Sysak11-Ischia10} observed that the identity map $\delta\colon C\longrightarrow K$ is a bijective derivation or $1$-cocycle with respect to the lambda action and that the
semidirect product $G=[K]C$ with respect to the lambda action possesses a triple factorisation $G=KC=KD=CD$, where $D=\{\delta(c)c\mid  c\in C\}$ is a subgroup of~$G$, and $K\cap C=K\cap D=C\cap D=1$. We will refer to it as the \emph{large trifactorised group} associated with~$B$, and denoted by $\mathsf{L}(B) = (G, K, D, C)$.

%\bigskip

It seems natural to consider the category $\mathbf{L3factGrp}$ whose objects are trifactorised groups of the form $(G, K, H, C)$ where $K\trianglelefteq G$, $H\le G$, $C\le G$, $G=KH=KC=HC$, $K\cap H=K\cap C=H\cap C=1$, and whose morphisms $f\colon (G_1, K_1, H_1, C_1)\longrightarrow (G_2,K_2,H_2, C_2)$ are group homomorphisms $f\colon G_1\longrightarrow G_2$ such that $f(K_1)\le K_2$, $f(H_1)\le H_1$, and $f(C_1)\le C_2$. If $f\colon B_1\longrightarrow B_2$ is a brace homomorphism and $\mathsf{L}(B_1)=(G_1, K_1, H_1, C_1)$ and $\mathsf{L}(B_2)=(G_2, K_2, H_2, C_2)$ are the large trifactorised groups corresponding to $B_1$ and $B_2$, respectively, then the induced group homomorphisms $K_1\longrightarrow K_2$ and $C_1\longrightarrow C_2$ induce a group homomorphism $G_1\longrightarrow G_2$ and, given an element $(\delta(c), c)\in H_1$, where $\delta\colon C_1\longrightarrow K_1$ is the identity map of~$B_1$, regarded as a bijective derivation associated with the lambda map, we have that $(f(\delta(c)),f(c))=(\bar\delta(f(c)), f(c))\in H_2$, where $\bar \delta\colon C_2\longrightarrow K_2$ is the bijective derivation associated with~$B_2$. If we call this homomorphism $\mathsf{L}(f)$, we have that $\mathsf{L}\colon \mathbf{SKB}\longrightarrow \mathbf{L3factGrp}$ becomes a functor.

Another natural approach to braces is by means of regular subgroups of the holomorph of the additive group. When this additive group corresponds to a vector space, we have regular subgroups of the affine group, analysed in \cite{CarantiDallaVoltaSala06,CatinoColazzoStefanelli15,CatinoRizzo09}.

%The following result shows the relation between braces with a given additive group and regular subgroups of its holomorph.

%\begin{teo}[{see \cite[Theorem~4.2 and Proposition~4.3]{GuarnieriVendramin17}}]\label{th-GV17}
 % Let $B$ be a brace with additive group $K=(B, {+})$. Then $H=\{(a,\lambda_a)\mid  a\in B\}$ is a regular subgroup of the holomorph $\Hol(K)=[K]{\Aut(K)}$ that is isomorphic to $(B, {\cdot})$. Conversely, if $(K, {+})$ is a group and $H$ is a regular subgroup of $\Hol(K)$, then $K$ becomes a  brace with $(K, {\cdot})\cong H$, where $ab=af(b)$ for $a$, $b\in K$ and $f\in\Aut(K)$, satisfies that $af\in H$.

  %As a consequence, there is a bijective correspondence between braces with additive group $(B, {+})$ and regular subgroups of $\Hol(B, {+})$. Furthermore, isomorphic brace structures with additive group $(B, {+})$ correspond to conjugate subgroups of $\Hol(B, {+})$ by elements of $\Aut(B, {+})$.

%As an indication of the interest of Theorem~\ref{th-GV17}, we can say that it  has been used as the natural tool to construct representatives of all isomorphism classes of braces with a given additive group. 
%\end{teo}

Let $B$ be a brace and let $K=(B, {+})$. By \cite[Theorem~4.2]{GuarnieriVendramin17}, $H=\{(a,\lambda_a)\mid  a\in B\}$ is a regular subgroup of $\Hol(K)$ that is isomorphic to $(B, {\cdot})$. If we consider the subgroup $S = KH \leq \Hol(K)$, it follows that $S=KH=KE=HE$,
where $E = \{(0,\lambda_b)\mid  b\in B\}$ and $\Cent_{E}(K)= K\cap E =H\cap E =1$. We call $\mathsf{S}(B)=(S, K, H, E)$ the \emph{small trifactorised group} associated with~$B$. 

\bigskip

%Given a regular subgroup $H$ of $\Hol(K)$, we can consider the subgroup $S$ of $\Hol(K)=[K]{\Aut(K)}$ generated by $K$ and $H$. Let $\pi\colon \Hol(K)\longrightarrow \Aut(K)$ be the projection given by $\pi(a,\omega)=\omega$. Then $S=KH=K\pi(H)=H\pi(H)$, $K\trianglelefteq S$, $\pi(H)\le S$, $\Cent_{\pi(H)}(K)=1$, and $K\cap\pi(H)=H\cap\pi(H)=1$, that is, the group $S$ admits a triple factorisation in which the additive group appears as a normal subgroup, the multiplicative group $H$ appears as another factor, and the third factor $\pi(H)$ corresponds to the automorphism group induced by~$H$ on $K$ by means of the lambda action. 
%Let $\sigma\colon H\longrightarrow K$ be the restriction of the ``projection'' of $G$ on the first component $K$, it is well known that $\sigma$ is a bijective derivation with respect to $\lambda\circ{\pi|}_H\colon H\longrightarrow \Aut(K)$.  We call $\mathsf{S}(B)=(S, K, H, E)$ the \emph{small trifactorised group} associated with~$B$.

As in the case of large trifactorised groups, we can consider the category $\mathbf{S3factGrp}$ of small trifactorised groups. The objects of this category are small trifactorised groups $(G, K, H, E)$ with $G$ a group, $K\trianglelefteq G$, $H\le G$, $E\le G$, $G=KH=KE=HE$, $K\cap E=H\cap E=1$, and $\Cent_E(K)=1$. As in $\mathbf{L3factGrp}$, a morphism $f\colon (G_1, K_1, H_1, E_1)\longrightarrow (G_2, K_2, H_2, E_2)$ of $\mathbf{S3factGrp}$ is a group homomorphism $f\colon G_1\longrightarrow G_2$ such that $f(K_1)\le K_2$, $f(H_1)\le H_2$, and $f(E_1)\le E_2$. However, in general, we cannot associate to a brace homomorphism $ f\colon B_1\longrightarrow B_2$, where $\mathsf{S}(B_1)=(G_1, K_1, H_1, E_1)$, $\mathsf{S}(B_2)=(G_2, K_2, H_2, E_2)$, a morphism $\bar f$ of $\mathbf{S3factGrp}$ such that $\bar f|_{K_1}$ coincides with the group homomorphism $K_1\longrightarrow K_2$ induced by $f$.
\begin{ex}\label{ex-C2-A5}
  Let $B_1$ be a trivial brace of order~$2$ and let $B_2$ be a brace with additive group isomorphic to the alternating group $A_5$ of degree~$5$ and in which the multiplication is the operation opposite to~$+$. We have that $\mathsf{S}(B_1)=(G_1, K_1, H_1, E_1)$ with $G_1=H_1=K_1$ cyclic of order~$2$ and $E_1=1$, and $\mathsf{S}(B_2)=(G_2, K_2, H_2, E_2)$ with $K_2\cong A_5$, $H_2=\{(t, \alpha_t^{-1})\mid  t\in A_5\}\cong A_5$, $E_2=\Inn(A_5)$ and $G_2=K_2H_2\cong A_5\times A_5$. It is clear that $f\colon B_1\longrightarrow B_2$ given by $f(a)=(1,2)(3,4)$, where $a$ is the generator of the additive group of $B_1$, defines a brace homomorphism, but there is no group homomorphism $f\colon G_1\longrightarrow G_2$ such that $f(K_1)\le K_2$, $f(H_1)\le H_2$ and $f(E_1)\le E_2$ with $f(a)=(1,2)(3,4)\in K_2$, because $f(a)\in f(K_1\cap H_1)\le K_2\cap H_2=1$.
\end{ex}

\section{A new category of trifactorised groups}\label{sec-3fact}
%We have seen that we can assign to each brace its large trifactorised group and its small trifactorised group. We can generalise both types of trifactorised groups.

Let $(B,+,\cdot)$ be a  brace with additive group $K$ and multiplicative group $C$. Let $\lambda\colon C\longrightarrow\Aut(K)$ be the corresponding lambda map and $\delta \colon C\longrightarrow K$ the identity map. Consider $\eta\colon C\longrightarrow E$ a group epimorphism such that $\ker\eta\leq\ker\lambda$. Then we can define the group homomorphism $\bar\lambda\colon E\longrightarrow \Aut(K)$ such that $\bar\lambda(\eta(c))=\lambda(c)$ for all $c\in C$. Hence $E$ acts on $K$ by automorphisms, and so we can consider the corresponding semidirect product $G=[K]E$.

Consider the map from $C$ to $G$ defined by $c\mapsto \delta(c)\eta(c)$, which is injective with image $H=\{\delta(c)\eta(c)\mid c\in C\}$. Moreover, given $c_1,c_2\in C$,
\begin{align*}
(\delta(c_1)\eta(c_1))(\delta(c_2)\eta(c_2))&=\delta(c_1)\bar\lambda_{\eta(c_1)}(\delta(c_2))\eta(c_1)\eta(c_2)\\&=\delta(c_1)\lambda_{c_1}(\delta(c_2))\eta(c_1)\eta(c_2)\\&=\delta(c_1c_2)\eta(c_1c_2).
\end{align*}
Therefore $H$ is a subgroup of $G$ isomorphic to $C$.

Every element $g\in G$ can be uniquely written as a product $g=\delta(c_1)\eta(c_2)$ for $c_1,c_2\in C$, and 
\begin{equation}\label{eq-3fact}
G=KE=KH=HE,\quad K\cap E=H\cap E=1.
\end{equation}

The map $\sigma\colon H\longrightarrow K$ given by $\sigma(\delta(c)\eta(c))=\delta(c)$ is bijective and 
\begin{align*}
\sigma((\delta(c_1)\eta(c_1))(\delta(c_2)\eta(c_2)))&=\sigma(\delta(c_1)\lambda_{c_1}(\delta(c_2))\eta(c_1c_2))\\&=\delta(c_1)\lambda_{c_1}(\delta(c_2))\\&= \sigma(\delta(c_1)\eta(c_1))\bar\lambda_{\eta(c_1)}(\sigma(\delta(c_2)\eta(c_2))).
\end{align*}
Therefore, $\sigma$ is a bijective derivation with respect to the action $\bar\lambda\circ \pi_E|_H\colon H\longrightarrow\Aut(K)$, where $\pi_E\colon G\longrightarrow E$ is the natural projection.

\begin{defi}
The tuple $(G,K,H,E)$ is  called a \emph{(generalised) trifactorised group}  associated with $(B,+,\cdot)$.
\end{defi}
We define the category \textbf{3factGrp} of (generalised) trifactorised groups, that has as objects tuples $(G,K,H,E)$ where $G$ is a group, $K$ is a normal subgroup of $G$, $H$ and $E$ are subgroups of $G$, and they satisfy~\eqref{eq-3fact}, and as morphisms $f\colon (G_1,K_1,H_1,E_1)\longrightarrow (G_2,K_2,H_2,E_2)$ group homomorphisms $f\colon G_1\longrightarrow G_2$ such that $f(K_1)\leq K_2$, $f(H_1)\leq H_2$, and $f(E_1)\leq E_2$. This category has as subcategories \textbf{L3factGrp} and \textbf{S3factGrp}. From now on when we talk about a trifactorised group we will mean a tuple in this category.

We have shown above that there are different trifactorised groups associated with a brace.  We will prove now that the converse is also true: given a trifactorised group $(G,K,H,E)$, we can associate a brace and a bijective derivation $\sigma\colon H\longrightarrow K$.

\begin{prop}\label{prop-BrAss3fact}
Let $(G,K,H,E)$ be a trifactorised group. Every element $g\in G$ can be uniquely written as a product $g=k_ge_g$, where $k_g\in K$ and $e_g\in E$. Then:
\begin{enumerate}
\item $H$ acts on $K$ via the action $\tilde\lambda$ given by $\tilde\lambda(h)(k)=\conj{e_h}{k}=e_hke_h^{-1}$ for $h\in H$, $k\in K$.
\item The map $\sigma\colon H\longrightarrow K$ given by $\sigma(h)=k_h$ for $h\in H$ is a bijective derivation with respect to~$\tilde\lambda$.
\item If we define the operation $\boxdot$ on $K$ as $k_1\boxdot k_2 =k_1({\conj{e_{k_1}}{k_2}})$ for $k_1$, $k_2\in K$, it follows that $(K,\cdot,\boxdot)$ is a skew left brace, with $(K,\boxdot)\cong H$ and lambda map $\tilde\lambda\circ\sigma^{-1}\colon (K,\boxdot)\longrightarrow\Aut(K)$.
\end{enumerate}
\end{prop}
\begin{proof}
  It is easy to check that $e_{h_1h_2}=e_{h_1}e_{h_2}$ for $h_1$, $h_2\in H$. This implies that $\tilde\lambda\colon H\longrightarrow\Aut(K)$ is an action of $H$ on $K$. Given $k\in K$, there exists $e\in E$ and $h\in H$ such that $k=he$, hence $ke^{-1}\in H$. Let us suppose that there exist $e_1$, $e_2\in E$ such that $ke_1$, $ke_2\in H$, then $e_1^{-1}e_2\in H\cap E=1$. Consequently, $\sigma\colon H\longrightarrow K$ is bijective. Furthermore, if $h_1=k_1e_1$, $h_2=k_2e_2\in H$ with $k_1$, $k_2\in K$ and $e_1$, $e_2\in E$, we have that
$\sigma(h_1h_2)=\sigma(k_1({\conj{e_1}{k_1}})e_1e_2)=k_1({\conj{e_1}{k_1}})=\sigma(h_1)\tilde\lambda(h_1)(\sigma(h_2))$.
That is, $\sigma$ is a bijective derivation with respect to $\tilde\lambda$.

By \cite[Proposition~1.11]{GuarnieriVendramin17}, we can define a second operation on $K$ such that $K$ becomes a brace, which is  $k_1\circ k_2=\sigma({\sigma^{-1}(k_1)\sigma^{-1}(k_2)})=k_1\boxdot k_2$. It is obvious that $\sigma^{-1}\colon (K,\boxdot)\longrightarrow H$ is a group isomorphism. The lambda map is $\lambda_{k_1}(k_2)=k_1^{-1}({k_1\boxdot k_2})=\conj{e_{k_1}}{k_2}=\tilde\lambda(k_1e_{k_1})(k_2)=(\tilde\lambda\circ \sigma^{-1}(k_1))(k_2)$.\qedhere
\end{proof}

\begin{defi}
Given a trifactorised group $(G,K,H,E)$, the brace defined in Proposition \ref{prop-BrAss3fact} is called \emph{the associated brace} of $(G,K,H,E)$ and it is denoted by $\B(G,K,H,E)$.
\end{defi}

\begin{prop}\label{prop-Br3factEquiv}
\begin{enumerate}
\item Every trifactorised group $(G,K,H,E)$ is a trifactorised group associated with the brace $\B(G,K,H,E)$.
\item Let $(B,+,\cdot)$ be a brace, and let $(G,K,H,E)$ be a trifactorised group associated with $(B,+,\cdot)$, then $\B(G,K,H,E)=(B,+,\cdot)$.
\end{enumerate}
\end{prop}
\begin{proof}
\begin{enumerate}
\item Following the notation of Proposition \ref{prop-BrAss3fact} and denoting $C=(K,\boxdot)$, we can consider $\eta=\pi_E|_H\circ\sigma^{-1}\colon C\longrightarrow E$, where $\pi_E\colon G\longrightarrow E$ is the natural projection. We have that $\eta$ is a group epimorphism, because $\sigma^{-1}$ is a group isomorphism by the argument of the last paragraph of the proof of Proposition~\ref{prop-Br3factEquiv} and $\pi_E$ is a group epimorphism. Moreover, $\ker\eta=\sigma({\ker\pi_E\cap H})=\sigma(K\cap H)=K\cap H$. If $k\in\ker\eta$, then $\lambda(k)=\tilde\lambda(k)=\id$, therefore $\ker\eta\leq\ker\lambda$. This epimorphism gives us a trifactorised group associated with $\B(G,K,H,E)$, which is exactly $(G,K,H,E)$.
\item Let $K$ and $C$ be the additive group and the multiplicative group of $B$ respectively and let $\lambda\colon C\longrightarrow \Aut(K)$ be the lambda map of $(B,+,\cdot)$. Write $\delta\colon C\longrightarrow K$ the identity map. Consider $\eta\colon C\longrightarrow E$ the epimorphism that corresponds to the associated trifactorised group $(G,K,H,E)$. The brace $\B(G,K,H,E)$, as a set, is $B$, and its additive operation is just the additive operation of $(B,+,\cdot)$. Furthermore
  \[\delta(c_1)\boxdot\delta(c_2)=\delta(c_1)({\conj{\eta(c_1)}{\delta(c_2)}})=\delta(c_1)\lambda_{c_1}(\delta(c_2))=\delta(c_1c_2).\]
This means that $(B,\boxdot)=(B,\cdot)$.\qedhere
\end{enumerate}
\end{proof}

\subsection{Morphisms of trifactorised groups}\label{subsec-3factMorph}

\begin{defi}
A \emph{morphism} $f\colon (G_1,K_1,H_1,E_1)\longrightarrow (G_2,K_2,H_2,E_2)$ between two trifactorised groups  is a group homomorphism $f\colon G_1\longrightarrow G_2$ such that $f(K_1)\leq K_2$, $f(H_1)\leq H_2$, and $f(E_1)\leq E_2$.
\end{defi}
These conditions gives us three group homomorphisms, namely $f|_{K_1}\colon K_1\longrightarrow K_2$, $f|_{H_1}\colon H_1\longrightarrow H_2$, $f|_{E_1}\colon E_1\longrightarrow E_2$, which are the restrictions of $f$ as group homomorphism. We prove that $f|_{K_1}$ and $f|_{H_1}$ are closely related.
\begin{lemma}\label{lemma-Morph-K-H}
Let $f\colon (G_1,K_1,H_1,E_1)\longrightarrow (G_2,K_2,H_2,E_2)$ be a trifactorised group morphism. Let $\sigma_i\colon H_i\longrightarrow K_i$ for $i=1$, $2$ be the corresponding bijective derivations. Then
\begin{enumerate}
\item $\sigma_2\circ f|_{H_1}=f|_{K_1}\circ\sigma_1$.
\item Given $h_1,h_2\in H_1$, $f(h_1)=f(h_2)$ if and only if $f(\sigma_1(h_1))=f(\sigma_1(h_2))$.
\end{enumerate}
\end{lemma}
\begin{proof}\begin{enumerate}
\item We use the notation on Proposition \ref{prop-BrAss3fact}. Statement~(1) follows directly from the fact the $f(h)$ is the product $f(k_h)f(e_h)$, where $f(k_h)\in K_2$, $f(e_h)\in E_2$, for all $h \in H$. Therefore, $\sigma_2(f(h))=f(k_h)=f({\sigma_1(h)})$, for all $h \in H$.
\item We have that $f(h_1)=f(h_2)$ if and only if $\sigma_2(f(h_1))=\sigma_2(f(h_2))$, equivalently, $f(\sigma_1(h_1))=f(\sigma_1(h_2))$.\qedhere
%we have that $f(h_i)=f(k_{h_i})f(e_{h_i})=f(\sigma_1(k_i))f(e_{h_i})$ for $i=1,2$. If $f(h_1)=f(h_2)$, since $f(\sigma_1(k_i))\in K_2$ and $f(e_{h_i})\in E_2$ for $i=1,2$, by the unicity of the factorisation, $f(\sigma_1(h_1))=f(\sigma_1(h_2))$. If $f(\sigma_1(h_1))=f(\sigma_1(h_2))$, it follows that $f(e_{h_1}^{-1}e_{h_2})=f(h_1^{-1}h_2)\in H\cap E=1$. Therefore, $f(e_{h_1})=f(e_{h_2})$, which implies that $f(h_1)=f(h_2)$.
\end{enumerate}
\end{proof}

Furthermore, the injectivity, surjectivity and bijectivity of $f|_{K_1}$, $f|_{H_1}$, $f|_{E_1}$, and $f$ are also closely related.

\begin{prop}\label{prop-3factMorphMonoEpiIso}
Let $f\colon (G_1,K_1,H_1,E_1)\longrightarrow (G_2,K_2,H_2,E_2)$ be a morphism, then:
\begin{enumerate}
\item If $f$ is a monomorphism (resp., epimorphism, isomorphism) then $f|_{K_1}$, $f|_{H_1}$, and $f|_{E_1}$ are monomorphisms (resp., epimorphisms, isomorphisms).
\item  $f|_{K_1}$ is a monomorphism (resp., epimorphism, isomorphism) if and only if $f|_{H_1}$ is a monomorphism (resp., epimorphism, isomorphism).
\item If $f|_{K_1}$ is an epimorphism, then $f$ is a epimorphism.\label{prop-3factMorphEpi}
\item If $f|_{K_1}$ and $f|_{E_1}$ are monomorphisms, then $f$ is a monomorphism.\label{prop-3factMorphMono}
\item If $f|_{K_1}$ and $f|_{E_1}$ are isomorphisms, then $f$ is a isomorphism. \label{prop-3factMorphIso}
\end{enumerate}
\end{prop}
\begin{proof}
\begin{enumerate}
\item If $f$ is injective obviously its restrictions are injective, so we only need to prove it for surjectivity and the result for bijectivity will follow. Suppose that $f$ is an epimorphism and let $k\in K_2$. There exists $g\in G_1$ such that $f(g)=k$. Then $g=k_1e_1$ for $k_1\in K_1$ and $e_1\in E_1$, hence $f(k_1)f(e_1)=k$. Therefore, $f(e_1)=f(k_1^{-1})k\in K_2\cap E_2=1$. Consequently, $f(k_1)=k$ and $f|_{K_1}$ is an epimorphism. We argue analogously for $H_1$ and $E_1$.
\item This follows from Lemma \ref{lemma-Morph-K-H}.
\item If $f|_{K_1}$ is an epimorphism, then $f|_{H_1}$ is an epimorphism too. Since $G_2=K_2H_2$, the result follows.
%Given $g_2\in G_2$ there exists $k_2\in K_2$ and $h_2\in H_2$ such that $g_2=k_2h_2$. There exists $k_1\in K_1$ and $h_1\in H_1$ such that $f(k_1)=k_2$ and $f(h_1)=h_2$. Then $f(k_1h_1)=k_2h_2=g_2$.
\item Let $g_1$, $g_2\in G_1$ such that $f(g_1)=f(g_2)$. Since $G_1=K_1E_1$ there exist $k_1$, $k_2\in K_1$ and $e_1$, $e_2\in E_1$ such that $g_1=k_1e_1$ and $g_2=k_2e_2$. Then $f(k_2^{-1}k_1)=f(e_2e_1^{-1})\in K_2\cap E_2=1$. Hence $f(k_1)=f(k_2)$ and $f(e_1)=f(e_2)$. By hypothesis, we have that $k_1=k_2$ and $e_1=e_2$. Thus, $g_1=g_2$ and $f$ is a monomorphism.
\item It follows from Statements~(\ref{prop-3factMorphEpi}) and~(\ref{prop-3factMorphMono}).\qedhere
\end{enumerate}
\end{proof}
If $f|_{K_1}$ is a monomorphism, then $f$ is not a monomorphism in general. 
\begin{ex}
Consider $G=\gen{x}\times \gen{y}$ with $x^n=y^n=1$ for some $n\in\N$, then $(G,\gen{x},\gen{xy},\gen{y})$ and $(\gen{x},\gen{x},\gen{x},1)$ are trifactorised groups. The projection $\pi$ from $G$ to $\gen{x}$ is a morphism of trifactorised groups such that $\pi|_{\gen{x}}$ and $\pi|_{\gen{xy}}$ are monomorphisms, but $\pi$ is not a monomorphism.
\end{ex} 

%Given a trifactorised group we can consider its associated brace, therefore, is natural to ask if a trifactorised group morphism induces a brace morphism.
\begin{prop}\label{prop-BrMorphAss}
Let $f\colon (G_1,K_1,H_1,E_1)\longrightarrow (G_2,K_2,H_2,E_2)$ be a trifactorised group morphism, then  $\B(f)=f|_{K_1}$ can be viewed as a brace morphism from $\B(G_1,K_1,H_1,E_1)$ to $\B(G_2,K_2,H_2,E_2)$.
\end{prop}
\begin{proof}
This follows directly from Lemma \ref{lemma-Morph-K-H} because $f$ is compatible with the derivations.
\end{proof}
\begin{remark}
If we use the language of categories, Proposition \ref{prop-BrMorphAss} shows that $\B\colon \textbf{3factGrp}\longrightarrow \textbf{SKB}$ is a functor. 
If $\widehat {\mathsf{B}}$ is the restriction of $\mathsf{B}$ to the category $\mathbf{L3factGrp}$ of large trifactorised groups, then it is easy to see that $\mathsf{L}\circ \widehat{\mathsf{B}}\colon \mathbf{L3factGrp}\longrightarrow \mathbf{L3factGrp}$ and $\widehat{\mathsf{B}}\circ \mathsf{L}\colon \mathbf{SKB}\longrightarrow \mathbf{SKB}$ are the identity functors. That is, $\mathbf{L3factGrp}$ and $\mathbf{SKB}$ are equivalent categories.
Proposition \ref{prop-3factMorphMonoEpiIso} shows that $\B$ preserves monomorphisms, epimorphisms and isomorphisms and reflects epimorphisms. Proposition \ref{prop-Br3factEquiv} shows that $\B$ is essentially surjective. Example \ref{ex-C2-A5} shows that $\B$ is not surjective in morphisms.
\end{remark}

We present a sufficient and necessary condition on a brace homomorphism $f$ for the existence of a trifactorised group morphism $\bar f$ such that $\B({\bar f})=f$.
\begin{prop}\label{prop-Brmorph-3fact-equiv-cond}
Let $f\colon (G_1,K_1,H_1,E_1)\longrightarrow (G_2,K_2,H_2,E_2)$ be a trifactorised group morphism. Let $\B(G_i,K_i,H_i,E_i)=(K_i,\cdot,\boxdot)$. Consider $\eta_i\colon(K_i,\boxdot)\longrightarrow E_i$ the epimorphism that constructs the trifactorised group $(G_i,K_i,H_i,E_i)$ for $i=1$, $2$. Then
\begin{equation}\label{eq-Brmorph-3fact-equiv-cond}
\B(f)(\ker\eta_1)\leq \ker\eta_2.
\end{equation}
\end{prop}
\begin{proof}
Let $k\in\ker\eta_1$, then $k\in K_1\cap H_1$, therefore $f(k)\in K_2\cap H_2$, and $\eta_2(f(k))=1$, which means that $\B(f)(k)\in\ker\eta_2$.
\end{proof}

If a brace homomorphism satisfies Equation~\eqref{eq-Brmorph-3fact-equiv-cond}, then it is the image of a trifactorised group morphism from $(G_1,K_1,H_1,E_1)$ to $(G_2,K_2,H_2,E_2)$.

\begin{prop}\label{prop-Brmorph-to-3factmorph}
For $i=1$, $2$ consider $(B_i,+,\cdot)$ a brace with additive group $K_i$, multiplicative group $C_i$, and lambda map $\lambda_i\colon C_i\longrightarrow\Aut(K_i)$. Consider $(G_i,K_i,H_i,E_i)$ a trifactorised group associated with $B_i$ with corresponding epimorphism $\eta_i\colon C_i\longrightarrow E_i$ for $i=1$, $2$. Let $f\colon B_1\longrightarrow B_2$ a brace morphism such that $f({\ker\eta_1})\leq\ker\eta_2$, then:
\begin{enumerate}
\item The map $\bar f\colon (G_1,K_1,H_1,E_1)\longrightarrow (G_2,K_2,H_2,E_2)$ given by \[\bar f({k\eta_1(c)})=f(k)\eta_2({f({c})})\] is a trifactorised group morphism.
\item $\B(\bar f)=f$.
\item If $f$ is an epimorphism, then $\bar f$ is an epimorphism.
\item $\bar f$ is a monomorphism (resp., isomorphism) if and only if $f$ is a monomorphism (resp., isomorphism), and $f({\ker\eta_1})=\ker\eta_2$.
\end{enumerate}
\end{prop}
\begin{proof}
Since $f$ satisfies Equation~\eqref{eq-Brmorph-3fact-equiv-cond}, $\bar f$ is well defined. Obviously, the restrictions of $\bar f$ to $K_1$ and $E_1$ are group homomorphisms, $\bar f(k\eta_1(c))=\bar f(k)\bar f({\eta_1(c)})$, and
$\bar f(\conj{\eta_1(c)}{k}))=\bar f({(\lambda_1)_c(k)})=f({{(\lambda_1)}_c(k)})={(\lambda_2)}_{f(c)}(f(k))=\conj{\eta_2(f(c))}{f(k)}=\conj{\bar f(\eta_1(c))}{\bar f(k)}$ for all $k\in K_1$ and $c\in C_1$. This proves that $\bar f$ is a trifactorised group morphism. Moreover, $\bar f$ restricted to $K_1$ coincides with $f$, therefore $\B(\bar f)=f$. Proposition \ref{prop-3factMorphMonoEpiIso} shows that if $\B({\bar f})=f$ is an epimorphism, then $\bar f$ is an epimorphism.

Suppose that $\bar f$ is a monomorphism. Then $\B(\bar f)=f$ is a monomorphism by Proposition~\ref{prop-3factMorphMonoEpiIso}. Let $c\in f^{-1}({\ker\eta_2})$, then $\bar f(\eta_1(c))=\eta_2(f(c))=1=\bar f(1)$. Since $\bar f$ is a monomorphism, $\eta_1(c)=1$ and $c\in\ker\eta_2$. Therefore, $\ker\eta_2\leq f(\ker\eta_1)$ and the other inclusion is given by hypothesis. Now suppose that $f$ is a monomorphism and $f(\ker\eta_1)=\ker\eta_2$. Let $c\in C_1$ such that $\eta_2(f(c))=\bar f({\eta_1(c)})=1$. We have that $f(c)\in\ker\eta_2=f(\ker\eta_1)$. Since $f$ is a monomorphism, it follows that $c\in\ker\eta_1$. This means that $\bar f|_{E_1}$ is a monomorphism. By hypothesis $\bar f|_{K_1}=f$ is a monomorphism. Therefore, by Proposition \ref{prop-3factMorphMonoEpiIso} $\bar f$ is a monomorphism.
\end{proof}

As we have seen in Proposition \ref{prop-Brmorph-3fact-equiv-cond}, given a trifactorised group morphism $f\colon (G_1,K_1,H_1,E_1)\longrightarrow (G_2,K_2,H_2,E_2)$, then $\B(f)$ satisfies the hypothesis of Proposition \ref{prop-Brmorph-to-3factmorph}, thus $\B(f)$ induces a trifactorised group morphism $\ol{\B(f)}\colon(G_1,K_1,H_1,E_1)\longrightarrow (G_2,K_2,H_2,E_2)$.
\begin{prop}\label{prop-Brmorph-to-3factmorph-bij}
Let $f\colon (G_1,K_1,H_1,E_1)\longrightarrow (G_2,K_2,H_2,E_2)$ be a trifactorised group morphism. Then $\ol{\B(f)}=f$.
\end{prop}
\begin{proof}
Let $\B(f)\colon(K_1,\cdot,\boxdot)\longrightarrow (K_2,\cdot,\boxdot)$. Obviously $\ol{\B(f)}|_{K_1}=f|_{K_1}$. For $i=1$, $2$, let $\eta_i\colon(K_i,\boxdot)\longrightarrow E_i$ be the epimorphism defining the trifactorised group $(G_i,K_i,H_i,E_i)$. Then $\ol{\B(f)}(\eta_1(k))=\eta_2({\B(f)(k)})=\eta_2(f(k))$. Since $f(k\eta_1(k))=f(k)f(\eta_1(k))\in H_2$, it follows that $\eta_2(f(k))=f({\eta_1(k)})$. Therefore, $\ol{\B(f)}|_{E_1}=f|_{E_1}$. Hence $\ol{\B(f)}=f$.
\end{proof}

Note that Proposition \ref{prop-Brmorph-to-3factmorph} extends the idea of Section \ref{sec-large-small} of associating to a brace homomorphism a morphism between the large or small trifactorised groups. 
\begin{cor}\label{cor-ext-Brmorph-large-small}
Let $f\colon (B_1,+,\cdot)\longrightarrow (B_2,+,\cdot)$ be a brace homomorphism. Let $C_i$ the multiplicative group, $K_i$ the additive group, and $\lambda_i\colon C_i\longrightarrow\Aut(K_i)$ the lambda map of $(B_i+,\cdot)$ for $i=1$, $2$. Then:
\begin{enumerate}
\item $f$ extends to a trifactorised group morphism between the large trifactorised groups. 
\item $f$ extends to a trifactorised group morphism between the small trifactorised groups if and only if $f(\ker\lambda_1)\leq\ker\lambda_2$. In particular, brace epimorphisms extend to small trifactorised groups morphisms.
\end{enumerate}
\end{cor}
\begin{proof}
It follows from Proposition \ref{prop-Brmorph-to-3factmorph}.
\end{proof}

\section{Isomorphism classes of the associated generalised trifactorised groups}\label{sec-Ass-3fact-IsoClass}
All the trifactorised groups associated with a brace $(B,+,\cdot)$ are constructed using group epimorphisms whose kernel contained in $\ker\lambda$. It is natural then to think that there is a relationship between two of them.
\begin{prop}\label{prop-3factAss-exists-epi}
Let $(B,+,\cdot)$ be a brace with additive group $K$, multiplicative group $C$, and lambda map $\lambda\colon C\longrightarrow\Aut(K)$. Consider $\eta_i\colon C\longrightarrow E_i$ a group epimorphism such that $\ker\eta_i\leq\ker\lambda$ with $(G_i,K,H_i,E_i)$ its corresponding trifactorised group associated with $(B,+,\cdot)$ for $i=1$, $2$.
\begin{enumerate}
\item If $\ker\eta_1\leq \ker\eta_2$, then there exists a trifactorised group epimorphism $f\colon (G_1,K,H_1,E_1)\longrightarrow (G_2,K,H_2,E_2)$ such that $f|_{K}$ is an isomorphism.
\item If $\ker\eta_1=\ker\eta_2$, then $(G_1,K,H_1,E_1)\cong (G_2,K,H_2,E_2)$.
\end{enumerate}
\end{prop}
\begin{proof}
Let $f\colon G_1\longrightarrow G_2$ given by $f(k\eta_1(c))=k\eta_2(c)$. Let $c_1,c_2\in C$ such that $\eta_1(c_1)=\eta_1(c_2)$, then $c_1^{-1}c_2\in\ker\eta_1\leq\ker\eta_2$. Hence, $\eta_2(c_1)=\eta_2(c_2)$ and $f$ is well defined.
If we denote by $\bar\lambda_i$ the action of $E_i$ on $K$ for $i=1$, $2$, we have that for all $k_1,k_2\in K$ and $c_1,c_2\in C$:
\begin{align*}
f({(k_1\eta_1(c_1))(k_2\eta_1(c_2))})&=f({({k_1+\bar\lambda_1({c_1})(k_2)})\eta_1(c_1c_2)})\\&=({k_1+\lambda_{c_1}(k_2)})\eta_2(c_1c_2).\\
f({k_1\eta_1(c_1)})=f({k_2\eta_1(c_2)})&=(k_1\eta_2(c_1))(k_2\eta_2(c_2))=({k_1+\bar\lambda_2({c_1})(k_2)})\eta_2(c_1c_2)\\&=({k_1+\lambda_{c_1}(k_2)})\eta_2(c_1c_2).
\end{align*}
Therefore, $f$ is a group homomorphism. Obviously, $f(K)=K$, $f(H_1)=H_2$, and $f(E_1)=E_2$, hence, $f$ is a trifactorised group morphism. By Proposition \ref{prop-3factMorphMonoEpiIso}, $f$ is surjective. Finally, $f$ restricted to $K$ is injective. Therefore, $f$ restricted to $K$ is an isomorphism. 

Moreover, it is easy to see that if $\ker\eta_1=\ker\eta_2$, then $f$ is a monomorphism, and therefore, $f$ is an isomorphism.\qedhere
\end{proof}

\begin{remark}
Proposition \ref{prop-3factAss-exists-epi} shows that all trifactorised groups associated with $(B,+,\cdot)$ are isomorphic to a trifactorised group of the form $(G,K,H,C/N)$, where $N$ is a normal subgroup of $C$ contained in $\ker\lambda$.
\end{remark}
 
%If the type of epimorphism that appears in Proposition \ref{prop-3factAss-exists-epi} exists between to trifactorised groups, then they have to be associated with the same brace.
\begin{prop}\label{prop-3factEpi-AssBr}
Let $f\colon (G_1,K_1,H_1,E_1)\longrightarrow (G_2,K_2,H_2,E_2)$ be an epimorphism of trifactorised groups such that restricted to $K_1$ is an isomorphism. Then $\B(G_1,K_1,H_1,E_1)$ and $\B(G_2,K_2,H_2,E_2)$ are isomorphic.
\end{prop}
\begin{proof}
This follows directly from the fact that $\B$ preserves isomorphisms.\qedhere
\end{proof}

Proposition \ref{prop-3factEpi-AssBr} shows that if two trifactorised groups are isomorphic, then they are associated with the same brace. This does not imply  that $\ker\eta_1=\ker\eta_2$, where $\eta_i$, $i=1$, $2$, are the epimorphisms defining the associated trifactorised groups. 

%We know that the converse is true, as stated in Proposition \ref{prop-3factAss-exists-epi}, but in general it is not true.
\begin{ex}
Consider $K_1=\gen{x}$, $K_2=\gen{y}$ where $x^2=y^2=1$. Let $B$ be the trivial brace of the group $K= K_1\times K_2$. Then $\lambda\colon K\longrightarrow\Aut(K)$ is trivial, therefore, $\ker\lambda=K$, so we can consider whichever epimorphism we want. Let $\eta_i\colon K_1\times K_2\longrightarrow K/K_i$ the natural epimorphism for $i=1$, $2$. It is obvious that $\ker\eta_1=K_1\neq K_2=\ker\eta_2$ and the trifactorised groups that they define are $(G_i,K,H_i,K/K_i)$, where $G_i=K_1\times K_2\times K/K_i$ for $i=1$, $2$. 

Let $f\colon (G_1,K,H_1,K/K_1)\longrightarrow (G_2,K,H_2,K/K_2)$ defined by
\begin{align*}
f(x,1,K_1)=(1,y,K_2),\; f(1,y,K_1)=(x,1,K_1),\; f(1,1,yK_1)=(1,1,xK_2).
\end{align*}
 It is easy to check that $f$ is a trifactorised group isomorphism. Therefore, $(G_1,K,H_1,K/K_1)\cong (G_2,K,H_2,K/K_2)$, but $\ker\eta_1\neq\ker\eta_2$.
\end{ex}

In the particular case of the large and small trifactorised groups the result is true.
\begin{prop}\label{prop-iso-large-small}
Let $(B,+,\cdot)$ be a brace with additive group $K$ and multiplicative group $C$. Let $(G,K,H,E)$ be a trifactorised group associated with $(B,+,\cdot)$ with corresponding epimorphism $\eta\colon C\longrightarrow E$. Then:
\begin{enumerate}
\item $(G,K,H,E)\cong \L(B,+,\cdot)$ if and only if $\ker\eta=1$.
\item $(G,K,H,E)\cong \S(B,+,\cdot)$ if and only if $\ker\eta=\ker\lambda$. 
\end{enumerate}
\end{prop}
\begin{proof}
If the kernels are equal, we have seen that the trifactorised groups are isomorphic, so we only need to prove the converse. Let us denote $\delta\colon C\longrightarrow K$ the identity map of $(B,+,\cdot)$.

If $(G,K,H,E)\cong \S(B,+,\cdot)$, by Proposition \ref{prop-3factIso-BrAut}, there exists a brace automorphism $f$ such that $f(\ker\lambda)=\ker\eta$. Since $\ker\lambda$ is invariant by brace automorphisms, it follows that $\ker\lambda=\ker\eta$. The argument is analogous for the large trifactorised group. 
\end{proof} 

%\begin{cor}\label{cor-3factAss-exists-epi-large-small}
%Let $(B,+,\cdot)$ be a brace and consider $(G,K,H,E)$ a trifactorised group associated with $B$. Then
%\begin{enumerate}
%\item There exists a trifactorised group morphism {\small $f\colon \L(B,+,\cdot)\longrightarrow (G,K,H,E)$} such that $f|_{K_1}$ is a group isomorphism.
%\item There exists a trifactorised group morphism {\small $f\colon (G,K,H,E)\longrightarrow \S(B,+,\cdot)$} such that $f|_{K}$ is a group isomorphism.
%\end{enumerate}
%\end{cor}
%\begin{proof}
%Is a direct consequence of Proposition \ref{prop-3factAss-exists-epi} and Proposition \ref{prop-iso-large-small}.
%\end{proof}

\begin{prop}\label{prop-3factIso-BrAut}
Let $(B,+,\cdot)$ be a brace with additive group $K$, multiplicative group $C$ and identity map $\delta\colon C\longrightarrow K$. If $(G_i,K,H_i,E_i)$ is a trifactorised group associated with $(B,+,\cdot)$ with the corresponding epimorphism $\eta_i\colon C\longrightarrow E_i$ for $i=1$, $2$. Then $(G_1,K,H_1,E_1)\cong (G_2,K,H_2,E_2)$ if and only if there exists $f\in\Aut(B,+,\cdot)$ such that $f(\ker\eta_1)=\ker\eta_2$.
\end{prop}
\begin{proof}
Let suppose that $(G_1,K,H_1,E_1)\cong (G_2,K,H_2,E_2)$ and let $g$ be a trifactorised group isomorphism. By Propositions \ref{prop-3factMorphMonoEpiIso} and \ref{prop-BrMorphAss}, $f=\B(g)\in\Aut(B,+,\cdot)$. By Proposition \ref{prop-Brmorph-3fact-equiv-cond} $f(\ker\eta_1)\leq\ker\eta_2$, and analogously with $g^{-1}$ we have that $\ker\eta_2\leq f({\ker\eta_1})$.

Suppose that there exists $f\in\Aut(B,+,\cdot)$ such that $f(\ker\eta_1)=\ker\eta_2$. Then $\ker\eta_1=\ker({\eta_2\circ f})$. By Proposition \ref{prop-Brmorph-to-3factmorph}, we can construct a trifactorised group isomorphism $\bar f\colon (G_1,K,H_1,E_1)\longrightarrow (G_2,K,H_2,E_2)$.\qedhere
\end{proof}

Let $(B,+,\cdot)$ be a brace with additive group $K$, multiplicative group $C$, and lambda map $\lambda$. Consider $f\in\Aut(B,+,\cdot)$ and $c\in\ker\lambda$, then $\lambda_{f(c)}(f(c'))=f({\lambda_c(c')})=f(c)$ for all $c\in B$.
Therefore, $f(c)\in\ker\lambda$ and $f({\ker\lambda})\leq\ker\lambda$. Arguing similarly with the inverse of $f$, the other inclusion holds. Hence $f_C(\ker\lambda)=\ker\lambda$.

Let $\Omega=\{N\unlhd C\mid N\leq\ker\lambda\}$. We have that $\Aut(B,+,\cdot)$ acts on $\Omega$ via $f\bullet N=f(N)$. The action is well defined because $f(N)\leq f(\ker\lambda)=\ker\lambda$.
\begin{teo}\label{teo-orb_3factAss}
Let $(B,+,\cdot)$ a brace with additive group $K$, multiplicative group $C$, and lambda map $\lambda\colon C\longrightarrow\Aut(K)$. Let $\Omega=\{N\unlhd C\mid N\leq\ker\lambda\}$. There is a bijection between the isomorphism classes of trifactorised groups associated with $(B,+,\cdot)$ and the orbits of the action of $\Aut(B,+,\cdot)$ on $\Omega$, which associates to the orbit of $N\in\Omega$ the trifactorised group associated with $(B,+,\cdot)$ corresponding to the natural epimorphism $\eta_N\colon C\longrightarrow C/N$.
\end{teo}
\begin{proof}
Let $N\in\Omega$. Since $\ker\eta_N=N\leq\ker\lambda$, $\eta_N$ defines a trifactorised group associated with $(B,+,\cdot)$. If we consider $N$, $M\in\Omega$ such that there exists $f\in\Aut(B,+,\cdot)$ which $f(N)=M$, by Proposition \ref{prop-3factIso-BrAut}, $\eta_N$ and $\eta_M$ define isomorphic trifactorised groups associated with $(B,+,\cdot)$. Hence this defines a map from the set of all orbits of the action of $\Aut(B,+,\cdot)$ on $\Omega$ to the set of all isomorphism classes of trifactorised groups associated with $(B,+,\cdot)$. This map is injective by Proposition \ref{prop-3factIso-BrAut}.

If $(G,K,H,E)$ is a trifactorised group associated with $(B,+,\cdot)$ constructed with the epimorphism $\eta\colon C\longrightarrow E$, by Proposition \ref{prop-3factAss-exists-epi}, we know that $\eta_{\ker\eta}$ gives us a trifactorised group isomorphic to $(G,K,H,E)$ associated with $(B,+,\cdot)$, therefore, our map is surjective too and  hence it is a bijection.\qedhere
\end{proof}

\section{Substructures of braces and generalised trifactorised groups}\label{sec-substruc}
In this section, we characterise substructures of braces in terms of generalised trifactorised groups.
%, in a similar way to what appears in \cite{BallesterEsteban22}. 
We will identify 
each of these substructures with the corresponding subgroup $L$ of the additive group, represented by the subgroup $K$, and the corresponding subgroup $\sigma^{-1}(L)$ of the multiplicative group, represented by the subgroup $H$.

It is possible to compute the images of subsets of $H$ in $K$ or the preimages of subsets of $K$ in $H$ by the associated derivation $\sigma$ by means of the following result (cf.~\cite[Lemma~3.3]{BallesterEsteban22}).

\begin{prop}
  Consider the trifactorised group $(G, K, H, E)$ and consider $\sigma\colon H\longrightarrow K$ the corresponding bijective derivation. Then:
  \begin{enumerate}
  \item If $L\subseteq K$, then $\sigma^{-1}(L)=LE\cap H$.
  \item If $S\subseteq H$, then $\sigma(S)=SE\cap K$.
  \end{enumerate}
\end{prop}
\begin{proof}\begin{enumerate}
  \item 
    Let $h\in \sigma^{-1}(L)$. Then $h=k_he_h$ with $\sigma(h)=k_h\in L$ and $h=k_he_h\in LE\cap H$. Conversely, suppose that for some $l\in L$ and $e\in E$, we have that $h=le\in LE\cap H$. Since $l\in K$ and $e\in E$, $\sigma(h)=k_h=l$ and $e_h=e$. Therefore $\sigma^{-1}(L)=LE\cap H$.
      
  \item Suppose that $k\in \sigma(S)$. Then there exists $h\in S$ such that $\sigma(h)=k$ and $h=ke_h$. Hence $k=(ke_h)e_h^{-1}\in SE\cap K$. Conversely, suppose that $se\in K$ with $s\in S$, $e\in E$. Then $se=(k_se_s)e\in K$, which implies that $e_se=1$ and so $\sigma(s)=k_s=se$. Consequently, $\sigma(S)=SE\cap K$.\qedhere
\end{enumerate}
\end{proof}

\begin{prop}\label{prop-substr-3fact}
Let $(B,+,\cdot)$ be a brace with additive group $K$ and multiplicative group $C$. Let $(G,K,H,E)$ be an associated trifactorised group to $(B,+,\cdot)$. Consider $L$ a subset of $B$, then:
\begin{enumerate}
\item The following  statements are equivalent:\label{en-sb}
\begin{enumerate}
\item $(L,+,\cdot)$ is a subbrace of $(B,+,\cdot)$.\label{en-sb-1}
\item $L$ is a subgroup of $K$ and $LE\cap H$ is a subgroup of $H$.\label{en-sb-2}
\item $L$ is a subgroup of $K$ and $LE\cap H$ is a subgroup of $\Norm_G(L)$.\label{en-sb-3}
\item $LE\cap LH$ is a subgroup of $G$.\label{en-sb-4}
\end{enumerate}
\item The following statements are equivalent:\label{en-li}
\begin{enumerate}
\item $L$ is a left ideal of $(B,+,\cdot)$.\label{en-li-1}
\item $L$ is a subgroup of $K$ and $E$ is a subgroup of $\Norm_G(L)$.\label{en-li-2}
\item $LE$ is a subgroup of $G$.\label{en-li-3}
\end{enumerate}
\item The following statements are equivalent:\label{en-sli}
\begin{enumerate}
\item $L$ is a strong left ideal of $(B,+,\cdot)$.\label{en-sli-1}
\item $L$ is a normal subgroup of $G$.\label{en-sli-2}
\end{enumerate}
\item The following statements are equivalent:\label{en-id}
\begin{enumerate}
\item $L$ is an ideal of $(B,+,\cdot)$.\label{en-id-1}
\item $L$ is a normal subgroup of $G$ and $LE\cap H$ is a normal subgroup of~$H$.\label{en-id-2}
\item $LE\cap LH$ is a normal subgroup of $G$.\label{en-id-3}
\end{enumerate}
\end{enumerate}
\end{prop}
\begin{proof}\begin{enumerate}
\item 

\ref{en-sb-1} if and only if \ref{en-sb-2}. $L$ is a subbrace if and only if is a subgroup of $K$ and $LE\cap H=\sigma^{-1}(L)$ is a subgroup of $H$.

\ref{en-sb-2} implies \ref{en-sb-3}. Let $le_l\in LE\cap H$, where $l\in L$ and $e_l\in E$ and consider $l_1\in L$. Then $(le_l)l_1(le_l)^{-1}=l(e_ll_1e_l^{-1})l^{-1}$. Since $le_l$, $l_1e_{l_1}\in LE\cap H$ we have that $l({e_ll_1e_l^{-1}})e_le_{l_1}=(le_l)(l_1e_{l_1})\in LE\cap H$. Therefore, $l({e_ll_1e_l^{-1}})\in L$, which means that $(le_l)l_1(le_l)^{-1}\in L$ because $L$ is a subgroup of $K$. Therefore, $LE\cap H$ normalises $L$.

\ref{en-sb-3} implies \ref{en-sb-4}. If $LE\cap H$ is a subgroup of $\Norm_G(L)$, then $L(LE\cap H)$ is a subgroup of $G$ and, by Dedekind's identity, $LE\cap LH$ is a subgroup of $G$.

\ref{en-sb-4} implies \ref{en-sb-2}. Note that
\[LE\cap LH\cap K=L(E\cap K)\cap LH=L\cap LH=L.\]
Hence, $L$ is a subgroup of $K$. Furthermore, $LE\cap LH\cap H=LE\cap H$ is a subgroup of $H$.

\item 

  \ref{en-li-1} if and only if \ref{en-li-2}. We have that $L$ is a left ideal of $B$ if and only if it is a subgroup of $K$ and is $\lambda$-invariant (recall that $\lambda$ corresponds to the action of $E$ on $K$). Therefore, $L$ is a left ideal if and only if it is a subgroup of $K$ that is invariant by the conjugation by elements of $E$, equivalently, $E\leq \Norm_G(L)$.

\ref{en-li-2} implies \ref{en-li-3}. Since $E$ normalises $L$ it follows that $LE$ is a subgroup.

\ref{en-li-3} implies \ref{en-li-2} Since $LE$ is a subgroup it follows that $L=L(K\cap E)=K\cap LE$ is a normal subgroup of $LE$. Therefore, $E$ normalises $L$.

\item 

\ref{en-sli-1} implies \ref{en-sli-2}. A strong left ideal is simply a left ideal which is normal in the additive group, therefore, we have that $L$ is a normal subgroup of $K$ and $E$ normalises $L$. Hence, $G=KE\leq \Norm_G(L)$, equivalently, $L$ is a normal subgroup of $G$.

\ref{en-sli-2} implies \ref{en-sli-1}. Since $L$ is a normal subgroup of $G$ we have that $LE$ is a subgroup of $G$, hence, by \ref{en-li}, it follows that $L$ is a left ideal. Furthermore, $L$ is normal in $K$, so $L$ is a strong left ideal.

\item 

\ref{en-id-1} if and only if \ref{en-id-2}. An ideal is a strong left ideal which is normal in the multiplicative group, therefore, $L$ is an ideal if and only if $L$ is a strong left ideal of $B$ and $LE\cap H=\sigma^{-1}(L)$ is a normal subgroup of $H$. By \ref{en-sli}, this is equivalent to $L$ being a normal subgroup of $G$ and $LE\cap H$ being a normal subgroup of $H$.

\ref{en-id-2} implies \ref{en-id-3}. Since $L$ is a normal subgroup of $G$ we have that $LE\cap LH=L(LE\cap H)$ is a subgroup of $G$. Since $L$ is normal in $G$ and $LE\cap H$ is normal in $H$ it follows that $LE\cap LH=L(LE\cap H)$ is normalised by $H$. We want to prove that it is also normalised by $E$.

Let us prove that $E\cap LH$ is a normal subgroup of $E$. First note that $K(LE\cap H)$ is normal in $KH=G$. Therefore, $K(LE\cap H)\cap E$ is a normal subgroup of $E$. Using Dedekind's identity we obtain that
\begin{align*}
K(LE\cap H)\cap E&=KL(LE\cap H)\cap E=K(LE\cap LH)\cap E\\&=KL(E\cap LH)\cap E=K(E\cap LH)\cap E\\&=(E\cap LH)(K\cap E)=E\cap LH
\end{align*} 
In conclusion, $E\cap LH$ is a normal subgroup of $E$. Since $L$ is normal in $G$ and $E\cap LH$ is normal in $E$, we have that $LE\cap LH=L(E\cap LH)$ is normalised by $E$. 

Therefore, $LE\cap LH$ is normalised by $H$ and $E$, hence, by $HE=G$.

\ref{en-id-3} implies \ref{en-id-2}. Since $LE\cap LH$ is normal in $G$, $LE\cap H=LE\cap LH\cap H$ is a normal subgroup of $H$. Furthermore, $LE\cap LH\cap K$ is a normal subgroup of $G$. As we have seen in \ref{en-sb}, $LE\cap LH\cap K=L$. Hence, $L$ is a normal subgroup of $G$.\qedhere
\end{enumerate}
\end{proof}

Our next result shows a natural embedding of the trifactorised group associated with the subbrace of a brace $B$ in the trifactorised group associated with $B$.

\begin{lemma}\label{lemma-3fact-epi-K-E}
  Consider $(G,K,H,E)$ a trifactorised group. Let $\sigma\colon H\longrightarrow K$ be the corresponding bijective derivation and let $\pi_E\colon G\longrightarrow E$ be the natural projection. If $L\subseteq K$, then $\pi_E(\sigma^{-1}(L))=L^{-1}H\cap E$.
\end{lemma}
\begin{proof}
  Let $h\in \sigma^{-1}(L)$, then $h=k_he_h$ with $k_h\in K$ and $e_h\in H$. Since $\sigma(h)=k_h$, we have that $k_h\in H$. Furthermore, $e_h=\pi_E(h)$. It follows that $e_h=k_h^{-1}h\in L^{-1}H\cap E$. We conclude that $\sigma^{-1}(L)\subseteq L^{-1}H\cap E$.

  Conversely, let $e=l^{-1}h\in L^{-1}H\cap E$. As before, let $h=k_he_h$ with $k_h\in K$ and $e_h\in E$. Then $e=l^{-1}k_he_h$ and so $l^{-1}k_h=1$, that is, $k_h=l$, and $e=e_h$. We conclude that $l=k_h=\sigma(h)$, in other words, $h\in \sigma^{-1}(L)$, and $\pi_E(h)=e$. Consequently, $e\in \pi_E(\sigma^{-1}(L))$. It follows that $L^{-1}H\cap E\subseteq \pi_E(\sigma^{-1}(L))$ and so we have the desired equality.
\end{proof}
% \begin{lemma}\label{lemma-3fact-epi-K-E}
% Consider $(G,K,H,E)$ a trifactorised group. If $\bar\eta\colon K\longrightarrow E$ is a group epimorphism such that $k\bar\eta(k)\in H$ for all $k\in K$ and $L\le K$, then $\bar\eta(L)=LH\cap E$.
% \end{lemma}
% \begin{proof}
% Let $l\in L$, then $\bar\eta(l)=l^{-1}(l\bar\eta(l))\in LH\cap E$. Hence $\bar\eta(L)\subseteq LH\cap E$. Conversely, let $lh\in LH\cap E$ and $h=k\bar\eta(k)$ where $l\in L$, $h\in H$, and $k\in K$. Therefore, $l(k\bar\eta(k))\in E$, which means that $k=l^{-1}$. Hence, $k\in L$ and $lh=\bar\eta(k)\in \bar\eta(L)$ as wanted. \qedhere
% \end{proof}
\begin{prop}\label{prop-3factAss-SubBr}
Let $(B,+,\cdot)$  be a brace with additive group $K$ and multiplicative group $C$. Let $\sigma\colon H\longrightarrow K$ be the bijective derivation and consider $(G,K,H,E)$ a trifactorised group associated with $(B,+,\cdot)$ with the corresponding epimorphism $\eta\colon C\longrightarrow E$. If $L\subseteq K$ such that $(L,+,\cdot)$ is a subbrace of $(B,+,\cdot)$ and $T=L\sigma^{-1}(L)$, then
\[(T,L,T\cap H,T\cap E)=(LE\cap LH,L, LE\cap H,LH\cap E)\]
is a trifactorised group associated with $(L,+,\cdot)$. Furthermore, if $\delta\colon C\longrightarrow K$ is the identity map and $D=\delta^{-1}(L)$, the trifactorised group is constructed by means of the epimorphism $\eta|_{D}\colon D\longrightarrow LH\cap E$.
\end{prop} 
\begin{proof}
By Dedekind's identity, $T=L\sigma^{-1}(L)=L(LE\cap H)=LE\cap LH$. Then $T\cap H=LE\cap H$ and $T\cap E=LH\cap E$.

By Proposition \ref{prop-substr-3fact}, $LE\cap LH$ is a group and $LE\cap H$ is a subgroup of $G$ that normalises $L$. Therefore $LE\cap LH=L(LE\cap H)\leq \Norm_G(L)$, thus $L$ is a normal subgroup of $LE\cap LH$. Using Dedekind's identity we have that
\begin{align*}
(LH\cap E)(LE\cap H)&=LE\cap (LH\cap E)H=LE\cap LH\cap EH=LE\cap LH.
\end{align*}
The other properties for $(T,L,T\cap H,T\cap E)$ to be a trifactorised group follow from an application of Dedekind's identity.

We have that $\eta|_{D}\colon D\longrightarrow\eta(D)$ is an epimorphism that defines a trifactorised group $(\tilde G,L,\tilde H,\eta(D))$ associated with $(L,+,\cdot)$. Consider the map $\beta\colon C\longrightarrow H$ given by $\beta(c)=\delta(c)\eta(c)$ for $c\in C$. Then $\delta=\sigma\circ\beta$ and $\eta={\pi_E|}_H\circ \beta$. Consequently, $\eta(D)=\eta(\delta^{-1}(L))=\pi_E(\beta(\beta^{-1}(\sigma^{-1}(L))))=\pi_E(\sigma^{-1}(L))=L^{-1}H\cap E=LH\cap E$ by Lemma~\ref{lemma-3fact-epi-K-E} and the fact that $L\le K$. % Then $\eta\circ\delta^{-1}\colon K\longrightarrow E$ is a group epimorphism such that  $k\eta\p{\delta^{-1}(k)}\in H$. Hence, applying Lemma~\ref{lemma-3fact-epi-K-E}, we obtain that $\eta(D)=\eta\p{\delta^{-1}(L)}=LH\cap E$. 
Hence, $\tilde G=L\eta(D)=L(LH\cap E)=T$, and $\tilde H=\tilde H(H\cap \tilde E)=H\cap\tilde H\tilde E=H\cap T$.
\end{proof}

\subsection{Generalised trifactorised subgroups}\label{subsec-sub3fact}

\begin{defi}
Given a trifactorised group $(G,K,H,E)$ we define a \emph{trifactorised subgroup} of $(G,K,H,E)$ as a trifactorised group $(\tilde G,\tilde K,\tilde H,\tilde E)$ such that $\tilde G\leq G$, $\tilde K\leq K$, $\tilde H\leq H$ and $\tilde E\leq E$. We write $(\tilde G,\tilde K,\tilde H,\tilde E)\leq (G,K,H,E)$.
\end{defi}

\begin{prop}\label{prop-sub3fact-intersec}
Let $(G,K,H,E)$ be a trifactorised group with bijective derivation $\sigma\colon H\longrightarrow K$ and let $(\tilde G,\tilde K,\tilde H, \tilde E)$ be a trifactorised subgroup of $(G,K,H,E)$ with bijective derivation $\tilde\sigma\colon \tilde H\longrightarrow\tilde K$. Then:
\begin{enumerate}
\item $\tilde K=\tilde G\cap K$, $\tilde H=\tilde G\cap H$ and $\tilde E=\tilde G\cap E$.
\item $\tilde\sigma=\sigma|_{\tilde H}$.
\end{enumerate}
\end{prop}
\begin{proof}
The proofs of all three equalities are analogous, so let us prove it for~$\tilde K$.
\[\tilde G\cap K=\tilde K\tilde E\cap K=\tilde K(\tilde E\cap K)=\tilde K.\]
Given $h\in \tilde H$, it admits a factorisation $h=\tilde{k}_h\tilde{e}_h$ with $\tilde{k}_h\in\tilde K\leq K$ and $\tilde{e}_h\in \tilde E\leq E$. Since this  factorisation is unique in $G$, we have that $\sigma(h)=\tilde{k}_h=\tilde{\sigma}(h)$.
\end{proof}
\begin{remark}
Proposition \ref{prop-sub3fact-intersec} shows that all trifactorised subgroups are of the form $(T,T\cap K,T\cap H,T\cap E)$ where $T$ is a subgroup of $G$, but not every subgroup of $G$ works, for example if $T=K$, we have $(K,K,K\cap H,1)$ which in general is not a trifactorised group because $\tilde H\tilde E=K\cap H$.
\end{remark}
\begin{defi}
Let $(G,K,H,E)$ be a trifactorised group. A subgroup $T$ of $G$ is a \emph{trifactorised subgroup} if $(T,T\cap K,T\cap H,T\cap E)$ is a trifactorised subgroup of $(G,K,H,E)$.
\end{defi}
Our next result characterises the trifactorised subgroups of $G$.
\begin{prop}\label{prop-sub3fact-equiv}
Let $(G,K,H,E)$ be a trifactorised group and let $T$ be a subgroup of $G$. The following statements are equivalent:
\begin{enumerate}
\item $T$ is a trifactorised subgroup of $G$. \label{sub3fact-equiv-1}
\item $T=(T\cap K)E\cap (T\cap K)H$.\label{sub3fact-equiv-2}
\end{enumerate}
\end{prop}
\begin{proof}
\ref{sub3fact-equiv-1} implies \ref{sub3fact-equiv-2}. Let $\sigma\colon H\longrightarrow K$ the bijective derivation of $(G,K,H,E)$ and $\tilde\sigma\colon T\cap H\longrightarrow T\cap K$ be the bijective derivation of the trifactorised subgroup $T$. By Proposition \ref{prop-sub3fact-intersec} \[(T\cap K)E\cap H=\sigma^{-1}(T\cap K)=\tilde{\sigma}^{-1}({T\cap K})=T\cap H.\] Then, by Dedekind's identity,
\begin{align*}
(T\cap K)E\cap (T\cap K)H=(T\cap K)({(T\cap K)E\cap H})=(T\cap K)(T\cap H)=T.
\end{align*}
\ref{sub3fact-equiv-2} implies \ref{sub3fact-equiv-1}. The equalities $(T\cap K)(T\cap H)=T$ and $(T\cap K)(T\cap E)=T$ follow directly from Dedekind's identity. Furthermore,
\begin{align*}
(T\cap H)(T\cap E)&=T\cap (T\cap H)E\\&=T\cap ((T\cap K)E\cap (T\cap K)H\cap H)E\\&=T\cap ((T\cap K)E\cap H)E\\&=T\cap (T\cap K)E\cap HE\\&=T\cap (T\cap K)E=T.
\end{align*}
Therefore $T$ is a trifactorised subgroup of $G$.\qedhere
\end{proof}

\begin{prop}
  Let $(B,+,\cdot)$ be a brace with additive group $K$ and multiplicative group $C$. Consider $(G,K,H,E)$ a trifactorised group associated with $(B,+,\cdot)$. Then the map $L\mapsto LH\cap LE$ is a bijection between subbraces of $(B,+,\cdot)$ and trifactorised subgroups of $(G,K,H,E)$. Furthermore, the restriction of this map to the set all ideals of $B$ defines a bijection between the set of all the ideals of $B$ and the set of all normal trifactorised subgroups of~$G$.
\end{prop}
\begin{proof}
If $L$ is a subbrace of $(B,+,\cdot)$, Proposition \ref{prop-3factAss-SubBr} shows that $LE\cap LH$ is a trifactorised subgroup of $(G,K,H,E)$. If $T$ is a trifactorised subgroup of $G$, we have that $T\cap K$ corresponds to a subbrace of $(B,+,\cdot)$ because $(T\cap K)H\cap (T\cap K)E=T\leq G$. Hence, $T\mapsto T\cap K$ is a map from the the set of all trifactorised subgroups of $(G,K,H,E)$ to the set of all subbraces of $(B,+,\cdot)$. It follows that these two maps are mutually inverses and so they are bijective. The last claim follows from Proposition~\ref{prop-substr-3fact}~(\ref{en-id}).
\end{proof}

\section{Quotients of generalised trifactorised groups}\label{sec-quo}
Our main goal in this section is to show how we can make quotients with the substructure introduced in Section \ref{subsec-sub3fact}, and how these quotients are related to brace quotients.

Note that not every quotient of a trifactorised group is a trifactorised group.
\begin{ex}
    If $(G,K,H,E)$ is a trifactorised group with $G\neq K$, then $(G/K,1,G/K,G/K)$ is not a trifactorised group because $G/K\cap G/K\neq 1$.
\end{ex}

Not all the quotients of a trifactorised group are quotients by a normal trifactorised subgroup.
\begin{ex}
Let $(G,K,H,E)$ be a trifactorised group with $E$ nontrivial and normal in $G$. Consider $T=E$. Then
\begin{align*}
KT\cap ET&=KE\cap EE=G\cap E=E,\\
HT\cap ET&=HE\cap EE=G\cap E=E,\\
(T\cap K)H\cap (T\cap K)E&=(E\cap K)H\cap(E\cap K)E=H\cap E=1
\end{align*}
The first two equations prove that $(G/T,KT/T,HT/T,ET/T)$ is a trifactorised group, but $T$ is not a trifactorised subgroup.
\end{ex}

The following theorem characterises the normal subgroups of a trifactorised group whose corresponding quotient is again trifactorised. 
\begin{teo}\label{teo-3fact-quo-equi}
Let $(G,K,H,E)$ be a trifactorised group and $T$ a normal subgroup of $G$, then the following statements are equivalent:
\begin{enumerate}
\item $(G/T,KT/T,HT/T,ET/T)$ is a trifactorised group. \label{quo-equi-1}
\item $T=(T\cap K)(T\cap E)=(T\cap H)(T\cap E)$.\label{quo-equi-2}
\item \label{quo-equi-3} $T=({(T\cap K)H\cap (T\cap K)E})(T\cap E).$
      \hfill 
      \puteqnum\label{eq-3fact-quo}
\end{enumerate}
\end{teo}
\begin{proof}
\ref{quo-equi-1} implies \ref{quo-equi-2}. By hypothesis $KT\cap ET=HT\cap ET=T$. Now,
\begin{align*}
(T\cap K)(T\cap E)&=(T\cap K)(KT\cap ET\cap E)=(T\cap K)(KT\cap E)\\&=KT\cap (T\cap K)E=KT\cap (KT\cap ET\cap K)E\\&=KT\cap (K\cap ET)E=KT\cap KE\cap ET\\&=KT\cap ET=T.
\end{align*}
Analogously, we have that $(T\cap H)(T\cap E)=T$.

\ref{quo-equi-2} implies \ref{quo-equi-3}.
\begin{align*}
((T\cap K)H\cap (T\cap K)E)(T\cap E)&=(T\cap K)H(T\cap E)\cap (T\cap K)E\\&=TH\cap (T\cap K)E=(T\cap K)({TH\cap E})\\&=(T\cap K)({(T\cap E)(T\cap H)H\cap E})\\&=(T\cap K)({(T\cap E)H\cap E})\\&=(T\cap K)(T\cap E)(H\cap E)=T.
\end{align*}

\ref{quo-equi-3} implies \ref{quo-equi-1}. 
\begin{align*}
HT&=H({(T\cap K)H(T\cap E)\cap (T\cap K)E})\\&=(T\cap K)H(T\cap E)\cap H(T\cap K)E\\&=(T\cap K)H(T\cap E)\cap G=(T\cap K)H(T\cap E).\\
ET&=({(T\cap K)H(T\cap E)\cap (T\cap K)E})E\\&=(T\cap K)H(T\cap E)E\cap (T\cap K)E\\&=G\cap (T\cap K)E=(T\cap K)E.
\end{align*}
Then $HT\cap ET=(T\cap K)H(T\cap E)\cap (T\cap K)E=T$. Furthermore $T=T\cap ET=T\cap (T\cap K)E=(T\cap K)(T\cap E)$. Therefore,
\begin{align*}
KT\cap ET&=K(T\cap K)(T\cap E)\cap E(T\cap K)(T\cap E)\\
&=K(T\cap E)\cap E(T\cap K)=(T\cap K)({K(T\cap E)\cap E})\\&=(T\cap K)(T\cap E)(K\cap E)=T.
\end{align*}
This proves that $KT/T\cap ET/T=T$ and $HT/T\cap ET/T=T$. Therefore $(G/T,KT/T,HT/T,ET/T)$ is a trifactorised group.
\end{proof}

\begin{cor}\label{cor-3fact-quo}
Let $(G,K,H,E)$ be a trifactorised group, let $\sigma\colon H\longrightarrow K$ be its associated bijective derivation, and let $T$ be a normal trifactorised subgroup of $(G,K,H,E)$. Then
\[(G/T,KT/T,HT/T,ET/T)\]
is a trifactorised group, with bijective derivation $\ol\sigma\colon HT/T\longrightarrow KT/T$ given by $\ol\sigma(hT)=\sigma(h)T$.
\end{cor}
\begin{proof}
By definition the equation \eqref{eq-3fact-quo} is satisfied, therefore, the quotient is a trifactorised group.

Given $g\in G$ its factorisation is $g=k_ge_g$ with $k_g\in K$ and $e_g\in E$, hence, $gT=(k_gT)(e_gT)$. Since the factorisation is unique, we have that $\ol\sigma(hT)=k_hT=\sigma(h)T$.
\end{proof}

\begin{remark}
Given a trifactorised group $(G,K,H,E)$ and a normal subgroup $T\leq G$, it follows that $(G/T,KT/T,HT/T,ET/T)$ is a trifactorised group if and only if $T$ is the kernel of some trifactorised group morphism. Therefore, given a trifactorised group morphism $f\colon (G_1,K_1,H_1,E_1)\longrightarrow (G_2,K_2,H_2,E_2)$, it follows that
\[\ker f=({(\ker f\cap K_1)H_1\cap({\ker f\cap K_1})E_1})({\ker f\cap E_1}).\]
\end{remark}

\begin{prop}\label{prop-3fact-Br-quo}
Let $(B,+,\cdot)$ be a brace with additive group $K$, multiplicative group $C$ and identity map $\delta\colon C\longrightarrow K$. Consider $(G,K,H,E)$ a trifactorised group associated with $(B,+,\cdot)$ with its corresponding epimorphism $\eta\colon C\longrightarrow E$. Let $I$ be an ideal of $B$ and suppose $I$ corresponds to the subgroup $L\leq K$. If $T=LH\cap LE$, then the trifactorised group $(G/T,KT/T,HT/T,ET/T)$ is associated with the brace $(B/I,+,\cdot)$ with epimorphism $\bar\eta\colon \ol C\longrightarrow\ol E$ given by $\bar\eta(cD)=\eta(c)T$ where $D=\delta^{-1}(L)$, $\ol C=C/D$ and $\ol E=ET/T$.
\end{prop}
\begin{proof}
  Note that the additive group of $B/I$ is $K/L=KT/T$ and the multiplicative group is $C/D$. The operations of the brace associated to the quotient are:
  \begin{align*}
\delta(c_1)T\delta(c_2)T&=\delta(c_1)\delta(c_2)T\\
\delta(c_1)T\boxdot\delta(c_2)T&=\ol\sigma({\ol\sigma^{-1}({\delta(c_1)T})\ol\sigma^{-1}({\delta(c_2)T})})\\&=\ol\sigma({({\sigma^{-1}({\delta(c_1)})T})({\sigma^{-1}({\delta(c_2)T})})})\\&=\sigma({\sigma^{-1}({\delta(c_1)})\sigma^{-1}({\delta(c_2)})})T\\&=\sigma({\delta(c_1)\delta(c_2)^{\eta(c_1)}\eta(c_1c_2)})T=\delta(c_1c_2)T
\end{align*}
These are the operations of $B/I$ viewed in its additive group $KT/T$. Therefore, $(G/T,KT/T,HT/T,ET/T)$ is associated with $B/I$.

The identity map $\ol\delta\colon C/D\longrightarrow KT/T$ of $B/I$ is given by $\ol\delta(cD)=\delta(c)T$. It follows that $\ol\delta(cD)\eta(c)T=\delta(c)\eta(c)T\in HT/T$, therefore, the epimorphism that constructs the trifactorised group $(G/T,KT/T,HT/T,ET/T)$ must be $\bar\eta\colon C/D\longrightarrow ET/T$ given by $\bar\eta(cD)=\eta(c)T$.\qedhere

%Firstly, notice that the additive group of $B/I$ is $K/L$ and the multiplicative group is $C/D$. Suppose that $c_1D=c_2D$, then $c_1^{-1}c_2\in D$, which implies that $\eta(c_1^{-1}c_2)\in\eta\p{\delta^{-1}(L)}$. By lemma \ref{lemma-3fact-epi-K-E}, $\eta\p{\delta^{-1}(L)}=LH\cap E=T\cap E\leq T$. Therefore, $\eta(c_1)T=\eta(c_2)T$ and $\bar\eta$ is well defined.

%If $\lambda\colon C\longrightarrow\Aut(K)$ is the lambda map of $B$, the lambda map of $B/I$ is $\bar\lambda\colon \ol C\longrightarrow \Aut(\ol K)$ given by $\bar\lambda_{cD}(k+L)=\lambda_c(k)+L$. We aim to prove that $\ker\bar\eta\leq\ker\bar\lambda$. Let $cD\in \ker \bar\eta$, then $\eta(c)\in T$. We want to prove that $\bar\lambda_{cD}(\delta(c_1)+L)=\delta(c_1)+L$ for all $c_1\in C$, which is equivalent to $\delta(c_1)^{-1}\lambda_c\p{\delta(c_1)}=\delta(c_1)^{-1}\delta(c_1)^{\eta(c)}\in L=T\cap K$, so we only need to prove that $\delta(c_1)^{-1}\delta(c_1)^{\eta(c)}\in T$. Since $\eta(c)$ is in $T$ and $T$ is normal in $G$, the result follows. So $\bar\eta$ defines a trifactorised group associated to $B/I$.
\end{proof}

%\begin{remark}
%Proposition \ref{cor-3fact-quo} show that given a brace epimorphism from $B_1$ to $B_2$ and a trifactorised group associated with $B_1$, there exists a trifactorised group associated with $B_2$ such that we can extend our epimorphism.
%\end{remark}

%We might think that when we make these quotients the large and small trifactorised groups associated with $(B,+,\cdot)$ give us the large and small trifactorised groups associated with the quotient brace. But this in not entirely true.
For the large trifactorised group associated with $(B,+,\cdot)$, we have:
\begin{prop}
Let $(B,+,\cdot)$ be a brace, denote $\L(B,+,\cdot)=(G,K,H,E)$ and consider $I$ an ideal of $B$ with $L$ its associated subgroup of $K$. If $T=LE\cap LH$, it follows that the trifactorised quotient of $G$ by $T$ is $\L({B/I,+,\cdot}).$
\end{prop}
\begin{proof}
We already know that $(G/T,KT/T,HT/T,ET/T)$ is a trifactorised group associated with $(B/I,+,\cdot)$. To prove that it is the large one, we need to prove that $KT\cap HT=T$.
\begin{align*}
KT\cap HT&=K(T\cap K)(T\cap H)\cap H(T\cap K)(T\cap H)\\
&=K(T\cap H)\cap H(T\cap K)=(T\cap K)({K(T\cap H)\cap H})\\&=(T\cap K)(K\cap H)(T\cap H)=(K\cap H)T=T.
\end{align*}
\qedhere
\end{proof}

The above proposition is false for the small trifactorised group associated with $(B,+,\cdot)$.

\begin{ex}
Consider the trifactorised group $(G,K,H,E)$ where $K=\gen{x}\cong C_6$, $E=\gen{y}\cong C_2$, $G=[K]E\cong D_{12}$ and $H=\gen{x^2, x^3y}\cong S_3$. Since $x^y=x^{-1}$ it follows that $\Cent_E(K)=1$, therefore, the trifactorised group is small.

Let $T=\gen{x^2}=H\cap K$. Since $\card{G}=12$ and $\card{T}=3$ it follows that $\card{G/T}=4$. In particular, $G/T$ is abelian. Hence, $\card{{\Cent_{ET/T}(K/T)}}=\card{ET/T}=\card{E}=2$. Thus $(G/T,K/T,H/T,ET/T)$ is not the small trifactorised group.
\end{ex}

We bring the paper to a close by showing that every trifactorised group associated to a brace $B$ is a quotient of the large trifactorised group and has the small trifactorised group as a quotient. 

%For making quotients with trifactorised groups we need normal subgroups $T\leq G$ that satisfies \eqref{eq-3fact-quo}. In that equation there are two factors $(T\cap K)H\cap (T\cap K)E$ and $T\cap E$. If one of them is equal to $T$, then we can make a trifactorised group quotient. If the first one is equal to $T$ we are talking about Corollary \ref{cor-3fact-quo}. Let us study the second one, $T\cap E=T$, which is equivalent to $T\leq E$.

\begin{prop}\label{prop-3fact-quo-E}
Let $(B,+,\cdot)$ be a brace with additive group $K$ and multiplicative group $C$. Consider $(G,K,H,E)$ a trifactorised group associated with $B$ with corresponding epimorphism $\eta\colon C\longrightarrow E$, and $T$ a normal subgroup of $G$ contained in $E$. Then $(G/T,KT/T,HT/T,E/T)$ is a trifactorised group associated with $B$ with its corresponding epimorphism $\bar\eta=\pi_T\circ \eta$ where $\pi_T\colon E\longrightarrow E/T$ is the natural epimorphism.
\end{prop}
\begin{proof}
Proposition \ref{teo-3fact-quo-equi} shows that $(G/T,KT/T,HT/T,E/T)$ is a trifactorised group. Analogously to \ref{prop-3fact-Br-quo} the operations of the associated brace are $\delta(c_1)T\delta(c_2)T=\delta(c_1)\delta(c_2)T$ and $\delta(c_1)T\boxdot\delta(c_2)T=\delta(c_1c_2)T$. But in this case we have that $KT/T\cong K$, therefore, we have that the associated brace is $(B,+,\cdot)$.

The elements of our trifactorised group are of the form $\delta(c_1)\eta(c_2)T$. Since its bijective derivation is $\ol\delta\colon C\longrightarrow KT/T$ given by $\ol\delta(c)=\delta(c)T$, we have that the epimorphism that corresponds to the trifactorised group is $\bar\eta\colon C\longrightarrow E/T$ given by $\bar\eta(c)=\eta(c)T=({\pi_T\circ\eta})(c)$.
\end{proof}

\begin{lemma}\label{lemma-3fact-epi-ker-E}
Let $f\colon (G_1,K_1,H_1,E_1)\longrightarrow (G_2,K_2,H_2,E_2)$ be a trifactorised group epimorphism such that $f|_{K_1}$ is an isomorphism. Then $\ker f\leq E$.
\end{lemma}
\begin{proof}
Let $ke\in\ker f$ with $k\in K_1$ and $e\in E_1$. Then $1=f(ke)=f(k)f(e)$. Since $f(k)\in K_2$ and $f(e)\in E_2$ we have that $f(k)=1$ and $f(e)=1$. Hence $f|_{K_1}$ is an isomorphism, and so, $k=1$ and $\ker f\leq E$.
\end{proof}

\begin{teo}\label{sql}
    Let $(B,+,\cdot)$ be a brace with additive group $K$ and multiplicative group $C$. Consider $(G_i,K_i,H_i,E_i)$ for $i=1$, $2$ two trifactorised groups associated with $B$ with corresponding epimorphisms $\eta_i\colon C\longrightarrow E_i$. If $\ker\eta_1\leq\ker\eta_2$, then $(G_2,K_2,H_2,E_2)$ is a quotient of $(G_1,K_1,H_1,E_1)$. In particular:
    \begin{enumerate}
        \item Every associated trifactorised group of $B$ is quotient of $\L(B,+,\cdot)$.
        \item Every associated trifactorised group of $B$ has $\S(B,+,\cdot)$ as a quotient.
    \end{enumerate}
\end{teo}
\begin{proof}
    The first part is a direct application of Proposition \ref{prop-3factAss-exists-epi}, Lemma \ref{lemma-3fact-epi-ker-E} and Proposition \ref{prop-3fact-quo-E}. The second part follows from the first one and Proposition \ref{prop-iso-large-small}.
\end{proof}

\section{Acknowledgements}

This work is supported by the grant CIAICO/2023/007 from the Conselleria d'Educació, Universitats i Ocupació, Generalitat Valenciana.

\bibliographystyle{plain}
\bibliography{bibgroup}
\end{document}